\documentclass[final]{siamltex}

\usepackage{amssymb,latexsym,amsmath}
\usepackage{epsfig}
\usepackage{caption}
\usepackage{caption}
 
 \usepackage{epsfig}

\usepackage{makeidx}

\allowdisplaybreaks
\def\Ebox#1#2{%
\medskip
\begin{center}
  \strut\epsfxsize=#1 \hsize\epsfbox{#2}
\end{center}
\smallbreak}

 \usepackage{graphicx,fancybox,latexsym,epsfig}
\usepackage{fancyhdr,amsmath,times,amsxtra,amssymb}
\usepackage{color}

\def\Ebox#1#2{%
\medskip
\begin{center}
  \strut\epsfxsize=#1 \hsize\epsfbox{#2}
\end{center}
\smallbreak}

\def\IR{I \kern-0.35em R\,}

\newcommand{\sy}[1]{{\color{black} #1}}

\def\l1{L_1}
\def\l2{L_2}

\newlength{\noteWidth}
\long\def\notes#1{\ifinner
           {\tiny #1}
           \else
           \marginpar{\barbox[t]{\noteWidth}{\raggedright\tiny #1}}
       \fi\typeout{#1}}

 \usepackage{color}

  \newtheorem{remark}{Remark}
  
  \newtheorem{example}{Example}

\fussy

\begin{document}

\title{A Universal Dynamic Program and Refined Existence Results for Decentralized Stochastic Control}

\author{Serdar Y\"uksel
\thanks{S. Y\"uksel is with the Department of Mathematics and
    Statistics, Queen's University, Kingston, Ontario, Canada, K7L
    3N6.  Email: yuksel@mast.queensu.ca. This research was
    partially supported by the Natural Sciences and Engineering Research Council of Canada (NSERC). Part of this work has been presented at the 2018 IEEE Conference on Decision and Control.}
}

\maketitle

\begin{abstract}
For sequential stochastic control problems with standard Borel measurement and control action spaces, we introduce a general (universally applicable) dynamic programming formulation, establish its well-posedness, and provide new existence results for optimal policies. Our dynamic program builds in part on Witsenhausen's standard form, but with a different formulation for the state, action, and transition dynamics. Using recent results on measurability properties of strategic measures in decentralized control, we obtain a standard Borel controlled Markov model. This allows for a well-defined dynamic programming recursion through universal measurability properties of the value functions for each time stage. In addition, new existence results are obtained for optimal policies in decentralized stochastic control. These state that for a static team with independent measurements, it suffices for the cost function to be continuous (only) in the actions for the existence of an optimal policy under mild compactness (or tightness) conditions. These also apply to dynamic teams which admit static reductions with independent measurements through a change of measure transformation. We show through a counterexample that weaker conditions may not lead to existence of an optimal team policy.  The paper's existence results generalize those previously reported in the literature. A summary of and comparison with previously reported results and some applications are presented.
\end{abstract}

\begin{AMS}
93E20, 90B99, 49J99 	
\end{AMS}

\section{Introduction}


An increasingly important research area of mathematical and practical interest is {\it decentralized stochastic control} or {\it team theory}, which involves multiple decision makers (DMs) who strive for a common goal but who have access only to local information. Applications include energy systems, the smart grid, sensor networks, and networked control systems, among others \cite{YukselBasarBook}. Few results are known regarding systematic methods to arrive at optimal solutions, and there exist problems (such as Witsenhausen's counterexample \cite{wit68}) which have defied solution attempts for more than 50 years. 


In this paper, we present a general (universally applicable) dynamic programming formulation for sequential decentralized stochastic control, establish the well-posedness of the dynamic program, and through this program and additional analysis involving decentralized strategic measures, obtain new existence results which generalize and complement previously reported results. 
 

It is worth noting that in the theory of commonly studied single DM stochastic control, there are few main techniques which are utilized to establish the existence and structure of optimal policies for finite horizon and infinite horizon problems. One such approach is based on {\it dynamic programming} and the corresponding {\it measurable selection criteria}, which has been summarized elaborately in \cite[Theorem 3.3.5 and Appendix D]{HernandezLermaMCP} and \cite{hernandezlasserre1999further} among many other references. Another method is based on the properties of {\it strategic measures} \cite{schal1975dynamic,piunovskii1998controlled,dynkin1979controlled,feinberg1996measurability}. These two techniques provide comprehensive sufficient conditions for existence results. One could also mention techniques based on empirical measures, which are particularly useful for infinite horizon problems, for both the average cost as well as discounted cost criteria, leading to linear programming/convex analytical conditions (e.g., \cite[Chp. 6]{HernandezLermaMCP}\cite{Bor02}). 

However, all of these approaches in general crucially require that the information is expanding over time, in the sense that the $\sigma$-field generated by the information available at the decision maker at time $n$ is a subset of that generated by the information at $n+1$ and so on. In decentralized stochastic control, this nestedness condition does not occur in general, and hence the classical methods are not typically applicable.

Our paper is closely related to Witsenhausen's standard form \cite{WitsenStandard}, which entails a general approach establishing that {\it any} sequential team optimization may admit a formulation appropriate for a dynamic programming analysis. However, well-posedness of a dynamic program or existence results were not established in \cite{WitsenStandard} for standard Borel (we recall that a metric space which is complete and separable is called a Polish space and a Borel subset of a Polish space is a {\it standard Borel} space \cite{srivastava2008course}) models and Witsenhausen notes that {\it ...a deeper understanding of these problems, including an existence theory for nonclassical patterns, will require more insight into some delicate questions in analysis...}. In this paper, one of our goals is to address some of the issues raised and left open by Witsenhausen. We present a detailed review of \cite{WitsenStandard} in Section \ref{WitStandSec} (prior to a comparison with the program proposed in our paper).

That dynamic programming can be a useful tool for a class of dynamic teams has been known since 1970s. Clearly, if all the information at any given decision maker is common knowledge between all decision makers, then the system can be viewed to be a centralized system and standard dynamic programming is applicable. However, if only some of the system variables are common knowledge, the remaining unknowns may or may not lead to a computationally tractable program generating an optimal solution. A possible approach toward establishing a tractable program is through the construction of a controlled Markov chain where the controlled Markov state may now live in a larger state space (for example a space of probability measures) and the actions are elements in possibly function spaces. This controlled Markov construction may lead to a computation of optimal policies. Such a {\it dynamic programming approach} has been adopted extensively in the literature (see for example, \cite{Athans}, \cite{yos75}, \cite{ChongAthans}, \cite{AicardiDavoli}, \cite{YukTAC09}, \cite{lamperski2015optimal} and significantly generalized and termed as the {\it common information approach} in \cite{NayyarMahajanTeneketzis} and \cite{NayyarBookChapter}) through the use of a team-policy which uses common information to generate partial functions for each DM to generate their actions using local information. This construction requires a common knowledge among decision makers, which is a unifying assumption in the aforementioned contributions in the literature. We refer the reader to \cite[Chapter 12]{YukselBasarBook} for a comprehensive discussion and some applications of this approach to zero-delay communications and networked control. On the other hand, the construction to be presented in this paper is applicable to an arbitrary sequential team. 

It is also worth mentioning related efforts in the stochastic control literature dealing with systems under partial information, e.g., which include \cite{rishel1970necessary,jacka2015informational}. These provide a dynamic programming formulation for a class of stochastic control problems with {\it perfect recall} leading to a generalized form for a Bellman (martingale) optimality principle for a class systems, see  \cite[Definition 4.1, Theorem 5.7]{jacka2015informational} (see also \cite{elliott1977optimal} for a controlled-Markovian and perfect-recall setup). The assumption that the information structure is classical (due to the perfect-recall property, see Section \ref{WitStandSec}) in these papers allows for explicit optimality equations for each sample path. Also along the dynamic programming based approaches, related contributions for a class of decentralized differential stochastic control systems (in continuous-time) are \cite{charalambous2016decentralized} and \cite{charalambous2017centralizedI} which have established the Bellman (martingale) and dynamic programming equations to arrive at optimality conditions. Notably, these results also establish conditions for the global optimality of person-by-person-optimal policies for a class of convex problems. 


With regard to the strategic measures approach, for classical stochastic control problems such measures were defined (see \cite{schal1975dynamic}, \cite{piunovskii1998controlled}, \cite{dynkin1979controlled} and \cite{feinberg1996measurability}) as the set of probability measures induced on the product (sequence) spaces of the state and action pairs by measurable control policies: Given an initial distribution on the state, and a policy, one can uniquely define a probability measure on the product space. Certain measurability, compactness and convexity properties of strategic measures for single decision maker problems were studied in \cite{dynkin1979controlled,piunovskii1998controlled,feinberg1996measurability,blackwell1976stochastic}. 

In \cite{YukselSaldiSICON17} strategic measures for decentralized stochastic control were introduced and many of their properties were established. For decentralized stochastic control problems, considering the set of strategic measures and compactification or convexification of these sets of measures through introducing private or common randomness allow for placing a useful topology, that of weak convergence of probability measures, on the strategy spaces. Combined with a compactness condition, this allowed for establishing the existence of optimal team policies in \cite{YukselSaldiSICON17}, where compactness, convexity and Borel measurability properties were established, which will be utilized later in the paper. It was shown in particular that, strategic measures for decentralized stochastic control can behave drastically different when compared with classical stochastic control problems. In addition to \cite{YukselSaldiSICON17}, in \cite{gupta2014existence} existence of optimal policies for static and a class of sequential dynamic teams have been studied, where conditions on the weak compactness of the sets of strategic measures were established. We provide a literature review on existence results in Section \ref{exis}.


\noindent{\bf Contributions.} In view of the review above, the paper makes two main contributions. 
\begin{itemize}
\item[(i)] For sequential stochastic control problems with standard Borel measurement and control action spaces, we introduce a dynamic programming formulation which is universally applicable, establish its well-posedness, and provide existence results for optimal policies. Our dynamic program builds in part on Witsenhausen's standard form, but with a different formulation for the state, action, and transition dynamics and optimality analysis. We show that the dynamic program is well-posed: Using recent results on measurability properties of strategic measures in decentralized control, we obtain a controlled Markov model with standard Borel state and state dependent action sets. This allows for well-defined dynamic programming recursions through universal measurability properties of the value functions for each time stage. 
\item[(ii)] We present existence results which significantly generalize those previously reported in the literature: These state that for a static team with independent measurements, it suffices for the cost function to be continuous (only) in the actions for the existence of an optimal policy under mild compactness conditions. These also apply to dynamic teams which admit static reductions with independent measurements through a change of measure transformation. Since continuity (only) in the actions cannot be in general relaxed, the existence condition cannot in general be made weaker. In particular, Theorems \ref{SuffCon2}, \ref{SuffCon2'''} and Theorem \ref{existenceRelaxed3} are the most general existence results, to our knowledge, for sequential team problems, as we elaborate on later in the paper.
\end{itemize}


%
%

\section{Supporting Results}

\subsection{Sequential dynamic teams and Witsenhausen's characterization of information structures}\label{witsenInfoStructureReview}

Witsenhausen's contributions (e.g., \cite{wit75,wit88,WitsenStandard}) to dynamic teams and characterization of information structures have been crucial in our understanding of dynamic teams. In this section, we introduce the characterizations as laid out by Witsenhausen, termed as {\it the Intrinsic Model} \cite{wit75}; see \cite{YukselBasarBook} for a more comprehensive overview and further characterizations and classifications of information structures. In this model (described in discrete time), any action applied at any given time is regarded as applied by an individual decision maker/agent, who acts only once. One advantage of this model, in addition to its generality, is that the characterizations regarding information structures can be concisely described.

Suppose that in the decentralized system considered below, there is a pre-defined order in which the decision makers act. Such systems are called {\it sequential systems} (for non-sequential teams, we refer the reader to Andersland and Teneketzis \cite{AnderslandTeneketzisI}, \cite{AnderslandTeneketzisII} and Teneketzis \cite{Teneketzis2}, in addition to Witsenhausen \cite{WitsenhausenSIAM71} and \cite[p. 113]{YukselBasarBook}). Suppose that in the following, the action and measurement spaces are standard Borel spaces, that is, Borel subsets of Polish (complete, separable and metric) spaces. In the context of a sequential system, the {\it Intrinsic Model} has the following components:

\begin{itemize}
\item A collection of {\it measurable spaces} $\{(\Omega, {\cal F}),
(\mathbb{U}^i,{\cal U}^i), (\mathbb{Y}^i,{\cal Y}^i), i \in {\cal N}\}$, with ${\cal N}:=\{1,2,\cdots,N\}$, specifying the system's distinguishable events, and the control and measurement spaces. Here $N=|{\cal N}|$ is the number of control actions taken, and each of these actions is taken by an individual (different) DM (hence, even a DM with perfect recall can be
regarded as a separate decision maker every time it acts). The pair $(\Omega, {\cal F})$ is a
measurable space (on which an underlying probability may be defined). The pair $(\mathbb{U}^i, {\cal U}^i)$
denotes the measurable space from which the action, $u^i$, of decision maker $i$ is selected. The pair $(\mathbb{Y}^i,{\cal Y}^i)$ denotes the measurable observation/measurement space for DM $i$.

\item A {\it measurement constraint} which establishes the connection between the observation variables and the system's distinguishable events. The $\mathbb{Y}^i$-valued observation variables are given by $y^i=\eta^i(\omega,{\bf u}^{[1,i-1]})$, ${\bf u}^{[1,i-1]}=\{u^k, k \leq i-1\}$, $\eta^i$ measurable functions and $u^k$ denotes the action of DM $k$. Hence, the information variable $y^i$ induces a $\sigma$-field, $\sigma({\cal I}^i)$ over $\Omega \times \prod_{k=1}^{i-1} \mathbb{U}^k$
\item A {\it design constraint} which restricts the set of admissible $N$-tuple control laws $\underline{\gamma}= \{\gamma^1, \gamma^2, \dots, \gamma^N\}$, also called
{\it designs} or {\it policies}, to the set of all
measurable control functions, so that $u^i = \gamma^i(y^i)$, with $y^i=\eta^i(\omega,{\bf u}^{[1,i-1]})$, and $\gamma^i,\eta^i$ measurable functions. Let $\Gamma^i$ denote the set of all admissible policies for DM $i$ and let ${\bf \Gamma} = \prod_{k} \Gamma^k$.
\end{itemize}

We note that, the intrinsic model of Witsenhausen gives a set-theoretic characterization of information fields, however, for standard Borel spaces, the model above is equivalent to that of Witsenhausen's.

One can also introduce a fourth component \cite{YukselBasarBook}:
\begin{itemize}
\item A {\it probability measure} $P$ defined on $(\Omega, {\cal F})$ which describes the measures on the random events in the model. 
\end{itemize}

Under this intrinsic model, a sequential team problem is {\it dynamic} if the
information available to at least one DM is affected by the action of at least one other DM. A decentralized problem is {\it static}, if the information available at every decision maker is only affected by exogenous disturbances (Nature); that is no other decision maker can affect the information at any given decision maker.

Information structures can also be classified as {\it classical}, {\it quasi-classical} or {\it nonclassical}. An Information Structure (IS) $\{y^i, 1 \leq i \leq N \}$ is {\it classical} if $y^i$ contains all of the information available to DM $k$ for $k < i$. An IS is {\it quasi-classical} or {\it partially nested}, if whenever $u^k$, for some $k < i$, affects $y^i$ through the measurement function $\eta^i$, $y^i$ contains $y^k$ (that is $\sigma(y^k) \subset \sigma(y^i)$). An IS which is not partially nested is {\it nonclassical}.

Let \[\underline{\gamma} = \{\gamma^1, \cdots, \gamma^N\}\]
and let a cost function be defined as:
\begin{eqnarray}\label{lossF}
J(\underline{\gamma}) = E^{\underline{\gamma}}[c(\omega_0,{\bf u})] = E[c(\omega_0,\gamma^1(y^1),\cdots,\gamma^N(y^N))],
\end{eqnarray}
for some non-negative measurable loss (or cost) function $c: \Omega_0 \times \prod_k \mathbb{U}^k \to \mathbb{R}_+$. Here, we have the notation ${\bf u}=\{u^t, t \in {\cal N}\}$, and $\omega_0$ may be viewed as the cost function relevant exogenous variable contained in $\omega$. 

\begin{definition}\label{Def:TB1}\index{Optimal team cost}
For a given stochastic team problem with a given information
structure, $\{J; \Gamma^i, i\in {\cal N}\}$, a policy (strategy) $N$-tuple
${\underline \gamma}^*:=({\gamma^1}^*,\ldots, {\gamma^N}^*)\in {\bf \Gamma}$ is
an {\it optimal team decision rule} ({\it team-optimal
decision rule} or simply {\it team-optimal solution}) if
\begin{equation}J({\underline \gamma}^*)=\inf_{{{\underline \gamma}}\in {{\bf \Gamma}}}
J({{\underline \gamma}})=:J^*. \label{eq:5}
\end{equation} 
The expected cost achieved by this strategy, $J^*$, is the optimal team cost.
\end{definition}

In the following, we will denote by bold letters the ensemble of random variables across the DMs; that is ${\bf y}=\{y^i, i=1,\cdots,N\}$ and ${\bf u}=\{u^i, i=1,\cdots,N\}$.

\subsection{Independent-measurements reduction of sequential teams under a new probability measure}\label{EquivIS}
 
Following Witsenhausen \cite[Eqn (4.2)]{wit88}, as reviewed in \cite[Section 3.7]{YukselBasarBook}, we say that two information structures are equivalent if: (i) The policy spaces are equivalent/isomorphic in the sense that policies under one information structure are realizable under the other information structure, (ii) the costs achieved under equivalent policies are identical almost surely, and (iii) if there are constraints in the admissible policies, the isomorphism among the policy spaces preserves the constraint conditions. 

A large class of sequential team problems admit an equivalent information structure which is static. This is called the {\it static reduction} of a dynamic team problem. For the analysis of our paper, we need to go beyond a static reduction, and we will need to make the measurements independent of each other as well as $\omega_0$. This is not possible for every team which admits a static reduction, for example quasi-classical team problems with LQG models \cite{HoChu} do not admit such a further reduction, since the measurements are partially nested. Witsenhausen refers to such an information structure as {\it independent static} in \cite[Section 4.2(e)]{wit88}.

%

Consider now (a static or a dynamic) team setting according to the intrinsic model where each DM $t$ measures $y^t=g_t(\omega_0,\omega_t,y^1,\ldots,y^{t-1},u^1,\ldots,u^{t-1})$, and the decisions are generated by $u^t=\gamma^t(y^t)$, with $1 \leq t \leq N$. Here $\omega_0,\omega_1,\cdots,\omega_N$ are primitive (exogenous) variables. We will indeed, for every $1 \leq n \leq N$, view the relation
 \[P(dy^n | \omega_0,y^1,y^2,\cdots,y^{n-1}, u^1,u^2,\cdots,u^{n-1}),\] as a (controlled) stochastic kernel (to be defined later), and through standard stochastic realization results (see \cite[Lemma 1.2]{gihman2012controlled} or \cite[Lemma 3.1]{BorkarRealization}), we can represent this kernel in a functional form through $y^n=g_n(\omega_0,\omega_n,y^1,y^2,\cdots,y^{n-1},u^1,u^2,\cdots,u^{n-1})$ for some independent $\omega_n$ and measurable $g_n$. 

This team admits an {\it independent-measurements} reduction provided that the following absolute continuity condition holds: For every $t \in {\cal N}$, there exists a reference probability measure $Q_t$ and a function $f_t$ such that for all Borel $S$:
\begin{eqnarray}
&& P(y^t \in S | \omega_0,u^1,u^2,\cdots,u^{t-1}, y^1,y^2,\cdots,y^{t-1}) \nonumber \\
&& \quad \quad = \int_{S} f_t(y^t,\omega_0,u^1,u^2,\cdots,u^{t-1},y^1,y^2,\cdots,y^{t-1}) Q_t(dy^t),
\end{eqnarray}

We can then write (since the action of each DM is determined by the measurement variables under a policy)
\begin{eqnarray}
&& P(d\omega_0,d{\bf y}, d{\bf u}) \nonumber \\
&&= P(d\omega_0) \prod_{t=1}^N \bigg(f_t(y^t,\omega_0,u^1,u^2,\cdots,u^{t-1},y^1,y^2,\cdots,y^{t-1}) Q_t(dy^t) 1_{\{\gamma^t(y^t) \in du^t\}}\bigg). \nonumber
\end{eqnarray}
The cost function $J(\underline{\gamma})$ can then be written as
\[J(\underline{\gamma})= \int P(d\omega_0) \prod_{t=1}^N (f_t(y^t,\omega_0,u^1,u^2,\cdots,u^{t-1},y^1,y^2,\cdots,y^{t-1}) Q_t(dy^t))  c(\omega_0,{\bf u}),\]
with $u^k = \gamma^k(y^k)$ for $1 \leq k \leq N$, and where now the measurement variables can be regarded as independent from each other, and also from $\omega_0$, and by incorporating the $\{f_t\}$ terms into $c$, we can obtain an equivalent {\it static team} problem. Hence, the essential step is to appropriately adjust the probability space and the cost function.

The new cost function may now explicitly depend on the measurement values, such that
\begin{eqnarray}
c_s(\omega_0,{\bf y}, {\bf u}) = c(\omega_0,{\bf u}) \prod_{t=1}^N f_t(y^t,\omega_0,u^1,u^2,\cdots,u^{t-1},y^1,y^2,\cdots,y^{t-1}). \label{c_sDefn}
\end{eqnarray}

Here we can reformulate even a static team to one which is, clearly still static, but now with independent measurements which are also independent from the cost relevant exogenous variable $\omega_0$. Such a condition is not restrictive; Witsenhausen's counterexample, to be reviewed later below in Section \ref{WitsenSection}, as well as the conditions required in \cite{gupta2014existence} satisfy this. 

We note that Witsenhausen, in \cite[Eqn (4.2)]{wit88}, considered a standard Borel setup; as Witsenhausen notes, a static reduction always holds when the measurement variables take values from countable set since a reference measure as in $Q_t$ above can be constructed on the measurement variable $y^t$ (e.g., $Q_t(z) = \sum_{i \geq 1} 2^{-i} 1_{\{z = m_i\}}$ where $\mathbb{Y}^t=\{m_i, i \in \mathbb{N}\}$) so that the absolute continuity condition always holds. We refer the reader to \cite{charalambous2016decentralized} for relations with classical continuous-time stochastic control where the relation with Girsanov's classical measure transformation is recognized, and \cite[p. 114]{YukselBasarBook} for further discussions.

\subsection{The strategic measures approach and a review of some related results}

In \cite{YukselSaldiSICON17}, strategic measures for decentralized stochastic control were studied. In this section, we review some of the results in \cite{YukselSaldiSICON17} that will be relevant in the dynamic programming approach to be presented in the following section. 

For single-DM stochastic control problems, {\it strategic measures} are defined (see \cite{schal1975dynamic}, \cite{piunovskii1998controlled}, \cite{dynkin1979controlled} and \cite{feinberg1996measurability}) as the set of probability measures induced on the sequence spaces of the state and action pairs by measurable control policies. In \cite{YukselSaldiSICON17}, through such a strategic measures formulation, existence, convexity and Borel measurability properties were established for decentralized stochastic control. It was shown in particular that, strategic measures for decentralized stochastic control can behave drastically different when compared with classical stochastic control problems.

\subsubsection{Sets of strategic measures for static teams}

Consider a static team problem defined under Witsenhausen's intrinsic model in Section \ref{witsenInfoStructureReview}. Let $L_A(\mu)$ be the set of strategic measures induced by all admissible team policies with $(\omega_0, {\bf y}) \sim \mu$. In the following, $B = B^0 \times \prod_{k} (A^k \times B^k)$ are used to denote the Borel sets in $\Omega_0 \times \prod_k (\mathbb{Y}^k \times \mathbb{U}^k)$,
\begin{eqnarray}
L_A(\mu)&:=& \bigg\{P \in {\cal P}\bigg(\Omega_0 \times \prod_{k=1}^N (\mathbb{Y}^k \times \mathbb{U}^k)\bigg): \nonumber \\
&& P(B) = \int_{B^0 \times \prod_k A^k} \mu(d\omega_0, d{\bf y}) \prod_k 1_{\{u^k = \gamma^k(y^k) \in B^k\}}, \nonumber \\
&& \qquad \qquad \qquad \qquad \gamma^k \in \Gamma^k, B \in {\cal B}(\Omega_0 \times \prod_k (\mathbb{Y}^k \times \mathbb{U}^k)) \bigg\} \nonumber
\end{eqnarray}
Let $L_A(\mu,\underline{\gamma})$ be the strategic measure induced by a particular $\underline{\gamma} \in {\bf \Gamma}$. 

Let $L_R(\mu)$ be the set of strategic measures induced by all admissible team policies where $\omega_0, {\bf y} \sim \mu$ and policies are individually randomized (that is, with independent randomizations):
\[L_R(\mu) := \bigg\{P \in {\cal P}\bigg(\Omega_0 \times \prod_{k=1}^N (\mathbb{Y}^k \times \mathbb{U}^k)\bigg): P(B) = \int_{B} \mu(d\omega_0, d{\bf y}) \prod_k \Pi^k(d u^k| y^k)\bigg\}\]
where $\Pi^k$ takes place from the set of stochastic kernels from $\mathbb{Y}^k$ to $\mathbb{U}^k$ for each $k$.

\subsubsection{Sets of strategic measures for dynamic teams}

We present the following characterization for strategic measures in dynamic sequential teams. Let for all $n \in {\cal N}$, 
\[h_n = \{\omega_0,y^1,u^1,\cdots,y^{n-1},u^{n-1},y^n,u^n\},\]
 and $p_n(dy^n|h_{n-1}) := P(dy^n | h_{n-1})$ be the transition kernel characterizing the measurements of DM $n$ according to the intrinsic model. We note that this may be obtained by the relation:
\begin{eqnarray}
&& p_n(y^n \in \cdot | \omega_0,y^1,u^1,\cdots,y^{n-1},u^{n-1}) \nonumber \\
&& \quad \quad := P\bigg(\eta^n(\omega, u^{1},\cdots,u^{n-1}) \in \cdot  \bigg| \omega_0,y^1,u^1,\cdots,y^{n-1},u^{n-1}\bigg) \label{kernelDefn}
\end{eqnarray}


Let $L_A(\mu)$ be the set of strategic measures induced by deterministic policies and let $L_R(\mu)$ be the set of strategic measures induced by randomized policies. We note as earlier that such an individual randomized policy can be represented in a functional form: By Lemma 1.2 in Gikhman and Shorodhod \cite{gihman2012controlled} and Theorem 1 in Feinberg \cite{Feinberg1}, for any stochastic kernel $\Pi^k$ from $\mathbb{Y}^k$ to $\mathbb{U}^k$, there exists a measurable function $\gamma^k: [0,1] \times \mathbb{Y}^k \to \mathbb{U}^k$ such that
\begin{eqnarray}\label{GSLemma}
m\{r: \gamma^k(r,y^k) \in A\} = \Pi^k(u^k \in A | y^k),
\end{eqnarray}
and $m$ is the uniform distribution (Lebesgue measure) on $[0,1]$. 

\begin{theorem}\label{StrategicCharacterization}\cite[Theorem 2.2]{YukselSaldiSICON17}
\begin{itemize}
\item[(i)] A probability measure $P \in {\cal P}\bigg(\Omega_0 \times \prod_{k=1}^N (\mathbb{Y}^k \times \mathbb{U}^k)\bigg)$ is a strategic measure induced by a deterministic policy (that is in $L_A(\mu)$) if and only if for every $n \in \{1,\cdots,N\}$:
\[\int P(dh_{n-1},dy^n) g(h_{n-1},y^{n}) = \int P(dh_{n-1}) \bigg(\int_{\mathbb{Y}^n} g(h_{n-1},z) p_n(dz | h_{n-1}) \bigg)\]
and
\[\int P(dh_n) g(h_{n-1},y^n,u^{n}) = \int P(dh_{n-1},dy^n) \bigg(\int_{\mathbb{U}^n} g(h_{n-1},y^{n},a) 1_{\{\gamma^n(y^n) \in da\}} \bigg)\]
for some $\gamma^n \in \Gamma^n$, for all continuous and bounded $g$, with $P(d\omega_0) = \mu(dw_0)$. 
\item[(ii)] A probability measure $P \in {\cal P}\bigg(\Omega_0 \times \prod_{k=1}^N (\mathbb{Y}^k \times \mathbb{U}^k)\bigg)$  is a strategic measure induced by a randomized policy (that is in $L_R(\mu)$)  if and only if for every $n \in \{1,\cdots,N\}$:
\begin{eqnarray}\label{convD1}
\int P(dh_{n-1},dy^n) g(h_{n-1},y^{n}) = \int P(dh_{n-1}) \bigg(\int_{\mathbb{Y}^n} g(h_{n-1},z) p_n(dz | h_{n-1}) \bigg) \nonumber \\
\end{eqnarray}
and
\begin{eqnarray}\label{convD2}
\int P(dh_n) g(h_{n-1},y^n,u^{n}) = \int P(dh_{n-1},dy^n) \bigg(\int_{\mathbb{U}^n} g(h_{n-1},y^{n},a^n) \Pi^n(da^n | y^n) \bigg) \nonumber \\
\end{eqnarray}
for some stochastic kernel $\Pi^n$ on $\mathbb{U}^n$ given $\mathbb{Y}^n$, for all continuous and bounded $g$, with $P(d\omega_0) = \mu(dw_0)$. 
\end{itemize}
\end{theorem}


The above will be useful for defining {\bf states} in decentralized stochastic control in the next section.

\subsubsection{Measurability properties of sets of strategic measures}
We have the following result, which will be crucial in the analysis to follow.
\begin{theorem}\label{FeinbergStrategicMeasurability3}\cite[Theorem 2.10]{YukselSaldiSICON17}
Consider a sequential (static or dynamic) team. 
\begin{itemize}
\item[(i)] The set of strategic measures $L_R(\mu)$ is Borel when viewed as a subset of the space of probability measures on $\Omega_0 \times \prod_k (\mathbb{Y}^k \times \mathbb{U}^k)$ under the topology of weak convergence.
\item[(ii)] The set of strategic measures $L_A(\mu)$ is Borel when viewed as a subset of the space of probability measures on $\Omega_0 \times \prod_k (\mathbb{Y}^k \times \mathbb{U}^k)$ under the topology of weak convergence.
\end{itemize}
\end{theorem}

For further properties of the sets of strategic measures, see \cite{YukselSaldiSICON17}.

\section{Controlled Markov Model Formulation for Decentralized Stochastic Control and Well-Posedness of a Dynamic Program}


\subsection{Witsenhausen's standard form \cite{WitsenStandard}}\label{WitStandSec}

Witsenhausen's standard form proceeds as follows, in the formulation of our paper.
\sy{
\begin{itemize}
\item[(i)] State: $x_1=\omega$, and for $2 \leq t \leq N$, $x_t = \{\omega, u^1,\ldots,u^{t-1} \}$.
\item[(i')] Extended State: For every Borel $B \in \Omega \times \prod_{i=1}^{t-1} \mathbb{U}^i$, \[\pi_t(B) = E_{\pi_t}[1_{ \{ (\omega, u^1,\ldots,u^{t-1}) \in B  \} } ]. \]
\item[(ii)] Control Action: $\gamma^t$ is $\sigma(y^t)$-measurable; $1 \leq t \leq N$.
\item[(iii)] (Total) Cost: $E[c(\omega,u^1,\ldots,u^N)]$.
\item[(iv)] Transition Dynamics/Kernel: 
\[x_t = g_t(x_{t-1};\gamma^{t-1}(x_{t-1})).\]
\end{itemize}}
As defined above, $\pi_t$ is ${\cal P}(\Omega \times \prod_{i=1}^{t-1} \mathbb{U}^i)$-valued. Let us endow this space of probability measures with the total variation norm. If one views the set of such probability measures as a subset of the set of all signed countably additive finite measures one could consider a duality pairing under the weak$^*$-topology \cite[Section 4.6]{DunfordSchwartz} (which would correspond to the classical weak convergence topology commonly studied in probability theory if the state and the measurement spaces were compact \cite{lue69}). However, even in the absence of such compactness conditions, with $\Psi_t$ being the set of bounded continuous functionals on $\Omega \times \prod_{i=1}^{t} \mathbb{U}^i$, $\psi_{t-1} \in \Psi_{t-1}$ defines a linear and continuous (bounded) map on the linear space of signed measures, where ${\cal P}(\Omega \times \prod_{i=1}^{t-1} \mathbb{U}^i)$ is viewed as a subset, through the relation
\[\langle \psi_{t-1}, \pi_{t} \rangle:= \int \pi_t(d\omega, du^1,\ldots,du^{t-1}) \psi_{t-1}(\omega, u^1,\ldots,u^{t-1}).\]
Here, $\psi_t$ can be viewed as a co-state variable defined by the relation
\begin{eqnarray}\label{recursionDual1}
 \langle \psi_t, [T_t(\gamma^t)]\pi_t \rangle =  \langle [T^*_t(\gamma^t)] \psi_t, \pi_t  \rangle 
 \end{eqnarray}
with
$T_t(\gamma^t): {\cal P}(\Omega \times \prod_{i=1}^{t-1} \mathbb{U}^i) \to {\cal P}(\Omega \times \prod_{i=1}^{t} \mathbb{U}^i)$ defined by
\[\pi_{t+1}(E) = [T_t(\gamma^{t})](\pi_t)(E) = \pi_{t}\bigg( \{x_{t}: g_t(x_{t};\gamma^{t}) \in E\} \bigg)\]
and $T^*_t: \Psi_t \to \Psi_{t-1}$ is adjoint to $T_t$ defined through (\ref{recursionDual1}):
\begin{eqnarray}\label{inductivePsi1}
[T^*_t(\gamma^t)]\psi_t = \psi_{t-1},
\end{eqnarray}
with the terminal condition
\begin{align}\label{termCWit}
\psi_N = c(\omega, u^1,\cdots,u^N),
\end{align}
and the (expected) cost induced by a team policy being
\[\langle \psi_N, [T_N(\gamma^N)]\pi_N \rangle.\]


\sy{Given the above, Witsenhausen states the following (necessity condition):
\begin{theorem}\label{standardDP}\cite[Thm. 1]{WitsenStandard}
If a team policy $\underline{\gamma}^*=\{\gamma^{1,*},\cdots,\gamma^{N,*}\}$ is optimal, then with $\pi^*_t$ being the associated state variables, we have that for all $1 \leq t \leq N$,
\begin{eqnarray}
J(\underline{\gamma}^*) = \min_{\gamma^t} \langle \psi^*_t, [T_t(\gamma^t)]\pi^*_{t} \rangle \label{WitOptEq}
\end{eqnarray}
where $\psi^*_t$ is obtained inductively through (\ref{inductivePsi1}) with $\gamma^t$ taken to be $\gamma^{t,*}$ and terminal condition given by (\ref{termCWit}).
\end{theorem}
\begin{proof} Write
\[\pi_t = [T_{t-1}(\gamma^{t-1})] \circ [T_{t-2}(\gamma^{t-2})] \circ \ldots [\circ T_1(\gamma^1)]\pi_0 \]
Here, $\pi_t$ does not depend on $\gamma^{t},\gamma^{t+1},\ldots,\gamma^N$. Likewise,
\[\psi_t = [T^*_{t+1}(\gamma^{t+1})] \circ [T^*_{t+2}(\gamma^{t+2})] \circ \ldots \circ [T^*_{N}(\gamma^{N})]\psi_N.\]
Thus, $\psi_t$ does not depend on $\gamma^1,\ldots,\gamma^{t}$. 
Thus, fixing all $\gamma^i, i \neq t$ to be from the set $\{\gamma^{1,*},\cdots,\gamma^{N,*}\} \setminus \{\gamma^{t,*}\}$, the optimal $\gamma^i$ has to be minimizing (\ref{WitOptEq}).
\end{proof}
}

\sy{

{\bf Dynamic programming using Witsenhausen's Standard Form.} We note that while Witsenhausen \cite[Thm. 1]{WitsenStandard} only gives a necessity condition of optimality, one could also obtain a sufficiency condition through backwards induction (by not fixing the control policies prior to any time $t$ to be optimal apriori). In particular, we can have the following reasoning: Through backwards induction as in Bellman's principle of optimality, first for $t=N$, for any $\gamma^t$, and admissible state $\pi_t$
\[ \langle \psi^*_{t}, [T_t(\gamma^{t,*})]\pi_t  \rangle =  \langle [T^*_t(\gamma^{t,*})]\psi^*_{t} , \pi_t \rangle \leq \langle \psi^*_t, [T_t(\gamma^{t})]\pi_{t} \rangle, \]
and with $\psi^*_{t-1} = [T^*_t(\gamma^{t,*})]\psi_{t}$, and dynamics $\pi_t = [T_{t-1}(\gamma^{t-1})]\pi_{t-1}$, we have that for $1 \leq t \leq N-1$, if the following holds,
\[ \langle \psi^*_{t-1}, [T_{t-1}(\gamma^{t-1,*})]\pi_{t-1}  \rangle \leq \langle \psi^*_{t-1}, [T_{t-1}(\gamma^{t-1})]\pi_{t-1}   \rangle, \]
then we have an optimal policy. By induction, the optimal cost would be equal to
\[\langle \psi^*_1, [T_1(\gamma^{1,*})]\pi_{1}  \rangle.\]
However, Witsenhausen's construction in \cite{WitsenStandard} does not address the well-posedness of such a dynamic program, that is, whether the recursions presented above are well-defined. Furthermore, the existence problem was not considered in \cite{WitsenStandard}. These are the main motivations behind the construction we present in the following section.
}

\subsection{A standard Borel controlled state and action space construction and  universal dynamic program}\label{newDPF}

We introduce a controlled Markov model based on the strategic measures formulation studied in \cite{YukselSaldiSICON17}, which will be utilized further below. This model allows for the recursions to be written explicitly and will lead to desirable measurability, and later existence, properties.

We split $\omega$ into $\omega_0$ and the rest (recall that $\omega$ represents the entire exogenous uncertainty acting on the system, including the initial state, the system and measurement noise): $\omega_0$ is the smallest variable whose sigma-field generated over $\Omega$, together with those generated by $u^1,\dots,u^N$ on their respective action spaces, $c$ is measurable on. Such an $\omega_0$ is typically explicitly given in a team problem formulation in the cost function, or can be attained by a reduction. 
\sy{
\begin{itemize}
\item[(i)] State: $x_t =  \{\omega_0,u^1,\cdots,u^{t-1},y^1,\cdots,y^{t}\}$, $1 \leq t \leq N$.
\item[(i')] Extended State: $\pi_t \in {\cal P}(\Omega_0 \times \prod_{i=1}^{t} \mathbb{Y}^i \times \prod_{i=1}^{t-1} \mathbb{U}^i)$ where, for Borel $B \in \Omega_0 \times \prod_{i=1}^{t} \mathbb{Y}^i \times \prod_{i=1}^{t-1} \mathbb{U}^i
$, \[\pi_t(B) := E_{\pi_t}[1_{ \{ (\omega_0,y^1,\cdots,y^t; u^1,\cdots,u^{t-1}) \in B  \} } ].\]
Thus, $\pi_t \in {\cal P}(\Omega_0 \times \prod_{i=1}^{t} \mathbb{Y}^i \times \prod_{i=1}^{t-1} \mathbb{U}^i)$ where the space of probability measures is endowed with the weak convergence topology.
\item[(ii)] Control Action: Given $\pi_t$, $\hat{\gamma}^t$ is a probability measure in ${\cal P}(\Omega_0 \times \prod_{k=1}^{t} \mathbb{Y}^k \times \prod_{k=1}^{t} \mathbb{U}^k)$ that satisfies the conditional independence relation: 
\[u^t \leftrightarrow y^t \leftrightarrow x_t = (\omega_0,y^{1},\cdots,y^{t}; u^1,\cdots,u^{t-1})\]
(that is, for every Borel $B \in \mathbb{U}^i$, almost surely under $\hat{\gamma}^t$, the following holds:
\[P(u^t \in B| y^t , (\omega_0,y^{1},\cdots,y^{t}; u^1,\cdots,u^{t-1})) = P(u^t \in B| y^t)\]
with the restriction
\[x_t \sim \pi_t.\]
Denote with $\Gamma^t(\pi_t)$ the set of all such probability measures. Any $\hat{\gamma}^t \in \Gamma^t(\pi_t)$ defines, for almost every realization $y^t$, a conditional probability measure on $\mathbb{U}^t$. 
When the notation does not lead to confusion, we will denote the action at time $t$ by $\gamma^t(du^t|y^t)$, which is understood to be consistent with $\hat{\gamma}^t$.
\item[(ii')] Alternative Control Action for Static Teams with Independent Measurements: Given $\pi_t$, $\hat{\gamma}^t$ is a probability measure on $\mathbb{Y}^t \times \mathbb{U}^t$ with a fixed marginal $P(dy^t)$ on $\mathbb{Y}^t$, that is $\pi^{\mathbb{Y}^t}_t(dy^t)=P(dy^t)$. Denote with $\Gamma^t(\pi_t^{\mathbb{Y}^t})$ the set of all such probability measures. As above, when the notation does not lead to confusion, we will denote the action at time $t$ by $\gamma^t(du^t|y^t)$, which is understood to be consistent with $\hat{\gamma}^t$.
\item[(iiii)] (Total) Cost: $E[c(\omega_0,u^1,\ldots,u^N)] =: \langle c, \hat{\gamma}^N(\pi_N) \rangle$, with:
\[\bigg(\hat{\gamma}^N(\pi_N) \bigg)(B):= E_{\hat{\gamma}^N(\pi_N)}[1_{ \{ (\omega_0,y^1,\cdots,y^N; u^1,\cdots,u^{t-1},u^N) \in B  \} } ],\]
for $1 \leq t \leq N$.
\item[(iv)] Transition Kernel: $T_t(\gamma^t): {\cal P}(\Omega_0 \times \prod_{i=1}^{t} \mathbb{Y}^i \times \prod_{i=1}^{t-1} \mathbb{U}^i) \to {\cal P}(\Omega_0  \times \prod_{i=1}^{t+1} \mathbb{Y}^i \times \prod_{i=1}^{t} \mathbb{U}^i)$ so that 
\[\pi_{t+1}=T_t(\gamma^{t})(\pi_t),\]
is defined by
\begin{eqnarray}\label{transKernelNew}
&&T_t(\gamma^{t})(\pi_t)(A_1,A_2,A_3)  \nonumber \\
&&= \int_{x_t \in A_1} \int_{u^t \in A_2} \int_{y^{t+1} \in A_3} P(dy^{t+1}|\omega_0,y^1,\cdots,y^t; u^1,\cdots,u^{t}) \gamma^t(du^t|y^t) \pi_t(dx_t),\nonumber
\end{eqnarray}
for Borel $A_1, A_2, A_3$. We recall again that $P(dy^{t+1}|\omega_0,y^1,\cdots,y^t; u^1,\cdots,u^{t})$ is given in (\ref{kernelDefn}), and the control actions $\gamma^{t}$ are defined above.
\end{itemize}}

\begin{remark}
\sy{With regard to items (ii)-(ii') in the controlled Markov model formulation, we caution that defining relaxed/randomized policies in decentralized stochastic control (as a way to facilitate the optimality analysis as in \cite{YukselSaldiSICON17}) is a subtle problem. One should not relax the probabilistic dependence and conditional independence properties among the decision makers too much: an illuminating example is what is known as the Clauser-Horne-Shimony-Holt game \cite{CHSH1969} (see also \cite{anantharam2007common} and \cite[Remark 1]{YukselSaldiSICON17} where a related convex class of information structures were introduced) in which the optimal cost under relaxed policies is strictly lower than the optimal cost for the original problem.} 
\end{remark}

Some further differences with Witsenhausen's standard form are as follows. 
\begin{itemize}
\item[(a)] In Witsenhausen's form, a deterministic update equation $x_t = g_t(x_{t-1};\gamma^{t-1}(x_{t-1}))$ exists, whereas in the formulation here we do not have such a measurable deterministic map, for $\omega$ is not part of the state. Furthermore, $\gamma^t$ is allowed to be randomized.
\item[(b)] Although it is not necessary for our analysis (we will apply a standard Bellman optimality argument with the aforementioned construction unlike Witsenhausen's approach), to be consistent with Witsenhausen's formulation, we can also define $\psi_t \in \Psi_t$ (as a bounded measurable functional from $\Omega_0 \times \prod_{k=1}^t \mathbb{Y}^{k+1} \times \prod_{k=1}^t \mathbb{U}^k$ to $\mathbb{R}$) as a co-state variable. First we can view the linear space of probability measures as a subset of the locally convex space \cite{rudin1991functional} of finite signed measures with the following notion of convergence: we say that $\nu_n \rightarrow \nu$ if $\int f(x) \nu_n(dx) \rightarrow \int f(x) \nu(dx)$ for every continuous and bounded function $f$. With the transition kernel given as (\ref{transKernelNew}),
we can define the recursions and the adjoint variable $\psi_t$ iteratively with
\begin{eqnarray}\label{recursionDual}
 \langle \psi_t, [T_t(\gamma^t)]\pi_t \rangle =  \langle [T^*_t(\gamma^t)]\psi_t, \pi_t  \rangle 
 \end{eqnarray}
where $T^*_t(\gamma^t): \Psi_t \to \Psi_{t-1}$ is adjoint to $T_t$ defined through (\ref{recursionDual}) (so that $[T^*_t(\gamma^t)]\psi_t = E^{\gamma^t}[\psi_t | x_t]$) -a more explicit representation will be presented further below using the properties of conditional expectation-, and
\begin{align}\label{FinalCondPsi}
\psi_N(\omega_0, y^1,\ldots,y^N, u^1,\ldots,u^N) := c(\omega_0, u^1,\ldots,u^N).
\end{align}
\item[(c)] Unlike Witsenhausen who endows the set of probability measures with total variation, we endow the set with the weak convergence topology. This we do to obtain a standard Borel space: total variation does not lead to a space of probability measures which is separable, whereas the space of probability measures on a complete, separable, metric (Polish) space endowed with the topology of weak convergence is itself a complete, separable, metric space \cite{Billingsley}. The Prokhorov metric, for example, can be used to metrize this space. 
\end{itemize}


In the proposed formulation, an explicit representation would be the following.
\begin{example}[Explicit representation as a terminal time problem] Given $\gamma^{t},\cdots,\gamma^N$, we have
\begin{eqnarray}
 &&\langle \psi_{t-1} , [T(\gamma^{t-1})]\pi_{t-1}  \rangle \nonumber \\
 &&=  \int \pi_{t-1}(d\omega_0,dy^1,\ldots,y^{t-1},du^1,\ldots,du^{t-2}) \nonumber \\
&& \quad \times E\bigg[ c\bigg(\omega_0,u^1,\ldots,u^{t-2},\gamma^{t-1}(y^{t-1}), \gamma^{t}(Y^t),\cdots,\gamma^N(Y^N)\bigg) \bigg| \omega_0,u^1,\cdots,u^{t-2},y^1,\cdots,y^{t-1}\bigg] \nonumber 
 \end{eqnarray}
Here, $\{Y^i\}$ variables are displayed with capital letters to highlight that these are random given $\omega_0$ (and $u^1,\cdots,u^{t-2},y^1,\cdots,y^{t-1}$), unlike what a corresponding expression would look like under Witsenhausen's formulation when they would be deterministic given $\omega$.
\end{example}


\sy{The following result is immediate by construction, through a standard Bellman optimality argument (see the analysis at the end of the previous section).
\begin{theorem}\label{standardDPRevised}
Consider a team policy $\underline{\gamma}^*=\{\gamma^{1,*},\cdots,\gamma^{N,*}\}$ that satisfies, for $1 \leq t \leq N$, and every admissible $\pi_t$,
\begin{eqnarray} \label{BellmanDPRevised}
J_t(\pi_t) = \inf_{\gamma^t} \langle \psi^*_t, [T_t(\gamma^t)]\pi_{t} \rangle
\end{eqnarray}
where $\psi^*_t$ is obtained inductively through (\ref{recursionDual}) with $\gamma^t$ taken to be $\gamma^{t,*}$ and final condition (\ref{FinalCondPsi}). Then, $\underline{\gamma}^*$ is optimal.
\end{theorem}
}
\subsubsection{Well-posedness of the proposed dynamic program in Theorem \ref{standardDPRevised}}

Witsenhausen does not address the well-posedness of a dynamic programming equation leading to the recursions given in (\ref{WitOptEq}), but notes that for teams with finite action and measurement spaces the equations would be well-defined. For the more general setup, as we will observe in the following, the well-posedness of the dynamic programming recursions requires us to revisit the concepts of universal measurability and analytic functions.

 
It suffices for a function to be universally measurable (and not necessarily Borel measurable) for its integration with respect to some probability measure to be well-defined: a function $f$ is $\mu$-measurable if there exists a Borel measurable function $g$ which agrees with $f$ $\mu$-a.e. A function that is $\mu$-measurable for every probability measure $\mu$ is called universally measurable. A set $B$ is universally measurable if the function $1_{\{x \in B\}}$ is universally measurable.

A measurable image of a Borel set is called an {\it analytic} set \cite[Appendix 2]{dynkin1979controlled}. We note that this is evidently equivalent to the seemingly more restrictive condition of being a {\it continuous} image of some Borel set. The image of a Borel set under a measurable function, and hence an analytic set, is universally measurable.

The integral of a universally measurable function is well-defined and is equal to the integral of a Borel measurable function which is $\mu$-almost equal to that function (with $\mu$ being the measure for integration). While applying dynamic programming, we often seek to establish the existence of measurable functions through the operation:
\[J_t(x_t) = \inf_{u \in \mathbb{U}(x_t)} \bigg( c(x_t,u) + \int J_{t+1}(x_{t+1})Q(dx_{t+1}|x_t,u) \bigg)\]
However, we need a stronger condition than universal measurability for the recursions to be well-defined: a function $f$ is called lower semi-analytic if $\{x: f(x) < c\}$ is analytic for each scalar $c \in \mathbb{R}$. Lower semi-analytic functions are universally measurable \cite{dynkin1979controlled}.

%


\begin{theorem}\cite{dynkin1979controlled}
Let $i:\mathbb{S} \to \mathbb{X}$ and $f: \mathbb{S} \to \mathbb{R}$ be Borel measurable functions. Then,
\[v(x) = \inf_{\{z: i(z) = x \}} f(z)\]
is lower semi-analytic.
\end{theorem}

Observe that (see p. 85 of \cite{dynkin1979controlled})
\[i^{-1}(\{x: v(x) < c\}) = \{z: f(z) < c\}\]
The set $\{z: f(z) < c\}$ is Borel, and thus if $i$ is also Borel, it follows that $v$ is lower semi-analytic. An application of this result leads to the following.

\begin{theorem}
Consider $G=\{(x,u): u \in \mathbb{U}(x) \subset \mathbb{U}\}$ a measurable subset of $\mathbb{X} \times \mathbb{U}$. Let $v: \mathbb{X} \times \mathbb{U} \to \mathbb{R}$ be Borel measurable. Then, the map
\[\kappa(x) := \inf_{(x,u) \in G} v(x,u),\]
is lower semi-analytic (and thus universally measurable).
\end{theorem}

The proof follows from the fact that $i^{-1}(\{x: \kappa(x) < c\}) = \{(x,z): v(x,z) < c\}$, where $i$ is the projection of $G$ onto $\mathbb{X}$, which is a continuous operation; the image may not be measurable but as a measurable mapping of a Borel set, it is analytic and thus the function $v$ is lower semi-analytic and thus universally measurable. The feasibility of dynamic programming follows: The next result is due to \cite[Theorems 7.47 and 7.48]{BertsekasShreve}.
\begin{theorem}\label{RecMeasLS}
Let $E_1, E_2$ be Borel sets, and let for every $x \in E_1$, $E_2(x)$ be a subset of $E_2$ such that the set $\{(x,u): x \in E_1, u \in E_2(x)\}$ is a Borel set in $E_1 \times E_2$. 
\begin{itemize}
\item[(i)] Let $g: E_1 \times E_2 \to \mathbb{R}$ be lower semi-analytic. Then, 
\[h(e_1) = \inf_{e_2 \in E_2(e_1)} g(e_1,e_2)\]
is lower semi-analytic.
\item[(ii)] Let $g: E_1 \times E_2 \to \mathbb{R}$ be lower semi-analytic. Let $Q(de_2 | e_1)$ be a stochastic kernel. Then, \[f(e_1): = \int g(e_1,e_2) Q(de_2|e_1)\]
is lower semi-analytic.
\end{itemize}
\end{theorem}
 


We now return to the question of the well-posedness of the dynamic programming formulation introduced in the previous subsection.

\begin{theorem}
The dynamic programming equations introduced in Theorem \ref{standardDPRevised} are well-defined.
\end{theorem}

\proof
%
{\bf Step (i).} We utilize the formulation introduced in Section \ref{newDPF} with the standard Borel action set $\hat{\gamma}^t$ and apply backwards induction. Let \[\pi_{N}(\cdots)= P(d\omega_0, dy^1,\cdots,dy^{N},du^1,\cdots,du^{N-1}).\] Recall that at time $N$, we minimize over $\hat{\gamma}^N$ which is the set of all probability measures which satisfies
\[u^N \leftrightarrow y^N \leftrightarrow (\omega_0,y^{1},\cdots,y^{N-1}; u^1,\cdots,u^{N-1})\]

%
We observe that the graph
\[\{ \pi_t, \Gamma^t(\pi_t)\},\] is measurable, that is $\{(\pi,\gamma^t(\pi))\}$, which is the set of all strategic measures that can be applied for every state and policy pairs is a measurable subset. Note that $\{(\pi_N,\gamma^N(\pi_N))\}$ is in fact $L_R(\mu)$, this is a Borel set by Theorem \ref{FeinbergStrategicMeasurability3}, and for each $1 \leq t \leq N$, we will have the restriction of this measurable set to $(\Omega_0,\prod_{k=1}^t \mathbb{Y}^k, \prod_{m=1}^{t-1} \mathbb{U}^m)$.

Theorem \ref{RecMeasLS}(i) thus ensures that
\[J_N^*(\pi_N) = \inf_{\gamma^N} \langle \psi^*_N, T_N(\gamma^{N})\pi_{N} \rangle\]
is a lower semi-analytic function

{\bf Step (ii).} For, $\inf_{\gamma^{N-1}} \langle \psi^*_{N-1}, T_{N-1}(\gamma^{N-1})\pi_{N-1} \rangle$, with
\begin{eqnarray}
\psi^*_{N-1}(x_{N-1}) &=& E\bigg[ c(\omega,u^1,\ldots,\gamma^{t-1}(y^{t-1}), \gamma^{N-1}(y^{N-1}),\gamma^{N,*}(Y^N)) \nonumber \\
&& \qquad \qquad \qquad \qquad \qquad \qquad \bigg| \omega_0,y^1,\ldots,y^N,u^1,\ldots,u^{N-1}\bigg] 
\end{eqnarray}
we also obtain a lower semi-analytic function. This follows because of the inductive nature of dynamic programming: Theorem \ref{RecMeasLS}(ii) ensures that $E[J(\pi_N) | \pi_{N-1},\gamma^{N-1}]$ is lower semi-analytic. Note that $P(d\pi_N | \pi_{N-1}, \gamma^{N-1})$ is a well-defined stochastic kernel. In particular, through the sequence of recursions
\begin{eqnarray}\label{seqDenk}
&& \inf_{\gamma^1} E^{\gamma^1}\bigg[  \inf_{\gamma^2} E^{\gamma^2} \bigg[  \nonumber \\
&&  \quad \quad  \cdots \inf_{\gamma^N} E^{\gamma^N} \bigg[ c(\omega_0,u^1,\cdots, u^N) \bigg]  \nonumber \\
&& \quad \quad \quad \quad \quad \quad \quad \quad \quad \quad \quad \quad \quad \quad \quad \quad \quad \quad \cdots  \bigg| y^2  \bigg| y^1 \bigg]
\end{eqnarray}
At each iteration, we apply first Theorem \ref{RecMeasLS} (i) and then (ii) to ensure lower semi-analyticity.

\endproof

\begin{remark}
The infima in (\ref{seqDenk}) can be replaced to be over deterministic $\gamma^n$ instead of $\hat{\gamma}^n$ due to \cite[Theorems 2.3 and 2.5]{YukselSaldiSICON17} and \cite[p. 1691]{gupta2014existence}.
\end{remark}

\section{Existence of Optimal Policies through the Dynamic Programming Formulation}

Given the dynamic programming formulation, the discussion above leads to the question on whether the dynamic programming approach can lead to new existence or regularity properties on optimal policies.

\subsection{A counterexample on non-existence under weak continuity of kernels and its analysis}\label{counterExampleSection}

In the following, we revisit an example presented in \cite[Chapter 3]{YukselBasarBook} on the non-existence of an optimal policy. The example from \cite{YukselBasarBook} has the same cost function and information structure as Witsenhausen's original counterexample, but with only discrete distributions. The linear dynamics are given as follows and the flow of dynamics and information are presented in Figures \ref{chap3figWitsenOrig} and \ref{chap3figWitsen}. 

\begin{figure}[h]
\centering
\Ebox{.50}{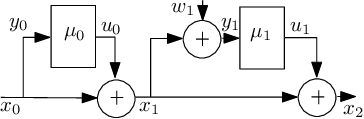}
\caption{Witsenhausen's counterexample.}\label{chap3figWitsenOrig}
\end{figure}

\[y_0=x_0,  \quad \quad u_0=\mu_0(y_0), \quad \quad x_1=x_0 + u_0,\]
\[ y_1=x_1 + w_1, \quad \quad u_1=\mu_1(y_1), \quad \quad x_2=x_1 + u_1.\]

The goal is to minimize the expected performance index for some $k>0$
\[Q_W(x_0,x_1, x_2, u_0, u_1) = k (u_0)^2 + x_2^2\]

Suppose $x_0$ and $w_1$ are two independent, zero-mean Gaussian random variables
with variance $\sigma^2$ and $1$. An equivalent representation is:
\[ u_0 = \gamma_0(x), \quad \quad u_1 = \gamma_1(u_0 + w) \]
\begin{equation}\label{eq: Wit-Q}
Q_W(x_0,x_1, x_2, u_0, u_1) = k (u_0 - x)^2 + (u_1 - u_0)^2
\end{equation}
This is the celebrated counterexample due to Witsenhausen, where he showed that an optimal policy for such a problem is non-linear: even though the system is linear, the exogenous variables are Gaussian, and the cost criterion is quadratic: decentralization makes even such a problem very challenging. We refer the reader to \cite[Chapter 3]{YukselBasarBook} and \cite{saldiyuksellinder2017finiteTeam} for a detailed historical discussion as well as on the numerical aspects for this very challening problem. 

Now, consider the same cost, but with the assumption that $x_0=x$ and $w_1$ are $\{-1,1\}$-valued independent uniformly distributed random variables.

For this setup, consider for each $\epsilon > 0$ sufficiently small, the construction:
$${\gamma_0}(x)=x + \epsilon\mbox{ sgn}(x)
\qquad
{\gamma_1}(y)=\left\{
\begin{array}{ll}
1+\epsilon, & \mbox{if } y = 2 + \epsilon   \mbox{ or  } \epsilon \\ \\ -1-\epsilon, & \mbox{if } y = -2 - \epsilon   \mbox{ or  } -\epsilon
\end{array} \right. \,,$$
and note that the corresponding value of $J$ is $k\epsilon^2$. This shows that $J$, can be made arbitrarily close to {\em zero} by picking $\epsilon$ sufficiently small.

Thus infimum of $J$ is $0$, but this cannot be achieved because in the limit as $\epsilon \to 0$, it is no longer possible to transmit $x$ (and thus $u_0$, which inevitably depends on it to keep the first term in $J$ small) perfectly. There does not exist an optimal policy!

\begin{figure}[h]
\centering
\Ebox{.45}{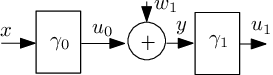}
\caption{Another representation of Witsenhausen's counterexample which makes the non-existence result for the discrete setup more evident.}\label{chap3figWitsen}
\end{figure}

\sy{The counterexample presented above can be formulated with the notation introduced in the paper}. In view of this, a formal negative result thus is the following, the proof of which is the counterexample presented above.
\begin{theorem}\label{weakNotSuff}
Consider a dynamic team. Suppose that \[E[\int f(z)\pi_{n+1}(dz)  | \pi_{n},\hat{\gamma}^n],\] is weakly continuous in $\pi_n,\hat{\gamma}^n$ for every continuous and bounded $f$. Suppose that all of the measurement and action spaces and $\Omega_0$ are compact and the cost function $c$ is continuous. Even for such a team, with only two decision makers, an optimal team policy may not exist.
\end{theorem}

We now go through the arguments to see why continuity/compactness conditions may not hold under such weak continuity conditions presented in Theorem \ref{weakNotSuff}.

Let $\pi_{N}(\cdots)= P(d\omega_0, dy^1,\cdots,dy^{N},du^1,\cdots,du^{N-1})$. Let $\hat{\gamma}^N$ be the joint measure defined by the control action $\gamma^N$. Note that the action space depends on the state. In particular, through an application of Theorem \ref{RecMeasLS}(i), by writing $e_1 = \pi_{N}$ and $e_2=\gamma^N(\pi_{N})$,  given that $E_2$ is Borel,
\[h(e_1) = \inf_{e_2 \in E_2} g(e_1,e_2)\]
will be lower semi-analytic. With some analysis one can show that the action space is compact. 


Through the compactness condition, we know that for every $\pi_N$ there exists an optimal $\gamma^N$. A question is whether $\gamma^{N}$ can be taken to be measurable in $\pi_{N}$. \sy{This holds because: $\Gamma^N(\pi_{N})$ is compact and its graph is measurable, as a consequence of the following result}:

%

\begin{lemma}\label{measSSchal} \cite[Theorem 2]{himmelberg1976optimal}, \cite{Schal} \cite{kuratowski1965general}
Let $\mathbb{X}, \mathbb{U}$ Polish spaces and $\Gamma= (x, \psi(x))$ where $\psi(x) \subset \mathbb{U}$ be such that, $\psi(x)$ is compact for each $x \in \mathbb{X}$ and $\Gamma$ is a Borel measurable set in $\mathbb{X} \times \mathbb{U}$. Let $c(x,u)$ be a continuous function on $\psi(x)$ for every $x$. Then,
there exists a Borel measurable function $f: \mathbb{X} \to \mathbb{U}$ such that
\[c(x,f(x)) = \min_{u \in \psi(x)} c(x,u) \]
\end{lemma}

By this result, it follows that $\hat{\gamma}^N$ can be measurably selected given $\pi_N$. Thus, an optimal policy exists for this time stage. Furthermore, $J_{N}(\pi_{N})$ is not only lower semi-analytic, but is actually Borel measurable (\sy{this also follows by realizing that the final stage decision needs to be an optimal estimator given the specifics of the problem}). 

However, $J_{N}(\pi_{N})$ is not in general continuous and it is this condition which fails for the counterexample. 

A sufficient condition for continuity is the following upper semi-continuity condition for the set-valued map $\Gamma^N(\pi_{N})$, mapping $\pi_N$ to the set $\Gamma^N(\pi_{N})$. For some related results see \cite[Appendix D, Proposition D.5]{HernandezLermaMCP}).

\begin{lemma}\label{suffCondUpperSemiCont}
$\inf_{\hat{\gamma}^N} J_N(\pi_N,\hat{\gamma}^N)$ is continuous in $\pi_N$ if (i) $\Gamma^N(\pi_{N})$ is an upper semi-continuous set-valued map so that if $\pi_{N,n} \to \pi_N$ and $\hat{\gamma}^{N,*}_n \to \hat{\gamma}^{N,*}$ with $\hat{\gamma}^{N,*}_n \in \Gamma^N(\pi_{N,n})$ then $\hat{\gamma}^{N,*} \in \Gamma^N(\pi_{N})$ and (ii) $\cup_{\pi_{N,n} \to \pi_N} \Gamma^N(\pi_{N,n})$ is pre-compact (that is, tight) for every sequence $\pi_{N,n} \to \pi_N$.
\end{lemma}
\begin{remark}
We note that condition (ii) in the statement of Lemma \ref{suffCondUpperSemiCont} above is sometimes imposed in the definition of upper semi-continuity of set-valued of maps. Nonetheless, for completeness, we have made this explicit as a further condition.
\end{remark}


For a class of static teams, this condition is applicable as we study more explicitly in the following.

\subsection{Existence for static teams or static-reducible dynamic teams}
In view of the discussion above, a positive result is the following; this significantly generalizes earlier reported results on existence. This result will be further generalized in Theorem \ref{existenceRelaxed2}, but the proof for the result below will be more direct and accessible. Furthermore, Theorem \ref{existenceControlTopology} will lead to a {\it robustness} result which shows that the value function of a sequential team problem may be continuous in the priors. We anticipate this result to have significant applications in decentralized stochastic control.

Note that for such static teams we modify the control action space $\Gamma^t$ at time $t$ (which for the static case is the set of probability measures on $\mathbb{Y}^t \times \mathbb{U}^t$ with a fixed marginal $P(dy^t)$ on $\mathbb{Y}^t$) as discussed in Section \ref{newDPF}.


\begin{theorem}\label{existenceControlTopology}
Consider a static or a dynamic team that admits a reduced form with independent measurements and $c_s$ (see (\ref{c_sDefn})) continuous. Suppose that all of the action spaces are compact and $c_s$ is continuous and bounded. Then, an optimal team policy exists. 
\end{theorem}

\proof {\bf Step (i).}
We represent the control action spaces with space of all joint measures on $\mathbb{Y}^i \times \mathbb{U}^i$ with a fixed marginal on $y^i$. Since the team is static, this decouples the action spaces from the actions of the previous decision makers. This also allows for the analysis studied earlier to be applicable and the value function at each time stage can be made to be continuous under weak convergence. 

{\bf Step (ii).} With the new formulation, we show that the optimal cost function $J_n(\pi_n)$ is continuous for every $1 \leq n \leq N$. This follows through an inductive argument as follows.
For the last stage $N$, this follows by the continuity of the cost function in $J_{N}(\pi_N,\gamma^N)$, as we establish below.

The key step in the proof is to show that we have the desired continuity under the new state and action space construction. In particular, let $\pi_{N-1,m} \to \pi_{N-1}$ and $(Q\gamma)^N_m \to (Q\gamma)^N$ (where convergence is in the weak sense). With an abuse of notation, suppose that $x_{N-1}$ also contains $u^{N-1}$ for the analysis below in {\bf Step (ii)}. Then, for every continuous and bounded $f$, with
\begin{eqnarray}
&& \int \pi_{N-1,m}(dx_{N-1}) Q(dy^N) \gamma^N_m(du^N|y^N) f(x_{N-1},y^N,u^N) = \int \pi_{N-1,m}(dx_{N-1}) h^m_{N-1}(x_{N-1}) \nonumber 
\end{eqnarray}

with
\[h^m_{N-1}(x) := \int \bigg( Q(dy^N) \gamma^N_m(du^N|y^N) f(x,y^N,u^N) \bigg),\]
and
\[h_{N-1}(x) := \int \bigg( Q(dy^N) \gamma^N(du^N|y^N) f(x,y^N,u^N) \bigg),\]
we show that $h^m_{N-1}(x_m)$ is a continuously converging function (as it is defined in \cite{serfozo1982convergence}) in the sense that for every $x_m \to x$,
\[h^m_{N-1}(x_m) \to h_{N-1}(x)\]
This follows from the following. First write 
\[h^m_{N-1}(x_m) - h_{N-1}(x) =  \bigg(h^m_{N-1}(x_m) - h_{N-1}(x_m)\bigg) + \bigg(h_{N-1}(x_m) - h_{N-1}(x)\bigg)\]
The second term converges to zero by the dominated convergence theorem. The first term converges to zero since
\begin{eqnarray}
&& \bigg(h^m_{N-1}(x_m) - h_{N-1}(x_m)\bigg) \nonumber \\
&& =  \int \bigg( Q(dy^N) \gamma^N_m(du^N|y^N) -  Q(dy^N) \gamma^N(du^N|y^N) \bigg) f(x_m,y^N,u^N) \nonumber
\end{eqnarray}

Here, $(Q\gamma)^N_m \to (Q\gamma^N)$ weakly and $f(x_m,y^N,u^N) \to f(x,y^N,u^N)$ continuously that is, as a mapping from $\mathbb{Y}^N \times \mathbb{U}^N \to \mathbb{R}$, $f_m(y^N,u^N):=f(x_m,y^N,u^N) \to f(y^N,u^N):=f(x,y^N,u^N)$ as $m \to \infty$. Thus, the convergence holds by a generalized convergence theorem given in \cite[Theorem 3.5]{serfozo1982convergence} or \cite[Theorem 3.5]{Lan81}.

Thus, continuity of $J_N(\pi_N, \gamma^N)$ in $(\pi_N, \gamma^N)$ holds.

{\bf Step (iii).} Now that we have established the continuity of $J_N(\pi_N, \gamma^N)$, it follows that by the compactness of $\Gamma^N$ (which is independent of the {\it state} of the system), 
\[J_N(\pi_N):=\inf_{\gamma^N \in \Gamma^N} J_N(\pi_N, \gamma^N)\]
is continuous in $\pi_{N}$. Furthermore, an optimal selection exists for $\gamma^N$ through the measurable selection condition given in Lemma \ref{measSSchal}.

{\bf Step  (iv).} In fact, this also implies that
\[J_{N-1}(\pi_{N-1},\gamma_{N-1}) = E[J_N(\pi_N) | \pi_{N-1},\gamma^{N-1} ], \]
is continuous in $\pi_{N-1}$ and $\gamma^{N-1}$. This follows because, $\pi_N$ is a deterministic function of $\pi_{N-1},\gamma^{N-1}$, and thus what is needed to be shown is that with $\pi_N = G(\pi_{N-1},\gamma^{N-1})$
\[\int \bigg(G(\pi_{N-1},\gamma^{N-1})\bigg)(dx_{N}) f(x_N) \]
is continuous in $\pi_{N-1},\gamma^{N-1}$. Note that 
\[\pi_N(dx_{N-1},du^{N-1},dy^N)=\pi_{N-1}(d\bar{x}_{N-1})(Q\gamma)^{N-1}(dy^{N-1},du^{N-1})Q^N(dy^N),\]
where we use the notation $\bar{x}_{N-1} = \{\omega_0,y^1,\ldots,y^{N-2},u^1,\ldots,u^{N-2}\}$ as $x_{N-1} \setminus {y^{N-1}}$. Through a similar argument as in {\bf Step (ii)}, continuity can be established.

{\bf Step (v).} By induction, $J_1(\pi_0)$ can be established. Furthermore, the optimal cost is continuous in $\pi_0$.
\endproof


We formally state the implication in {\bf Step (v)} as a theorem.

\begin{theorem}
Under the hypotheses of Theorem \ref{existenceControlTopology}, the optimal team cost is continuous in $\pi_0$, the prior measure on $(\omega_0, y^1)$.
\end{theorem}

\section{Existence of Optimal Policies in Decentralized Stochastic Control and Comparison with the Literature}\label{exis}

\subsection{New existence results in the most general form}
In the following, we first present the new existence results of the paper in the most general form. The first one follows from the dynamic programming method presented earlier, and the second one builds on a strategic measures approach. 


\begin{theorem}\label{existenceRelaxed2}
Consider a static or a dynamic team that admits a reduced form with independent measurements and $c_s$ (see (\ref{c_sDefn})) continuous. Suppose further that $\mathbb{U}^i$ is $\sigma-$compact (that is, $\mathbb{U}^i = \cup_n K_n$ for a countable collection of increasing compact sets $K_n$) and without any loss, the control laws can be restricted to those with $E[\phi^i(u^i)]\leq M$ for some lower semi-continuous $\phi^i: \mathbb{U}^i \to \mathbb{R}_+$ which satisfies $\lim_{n \to \infty} \inf_{u^i \notin K_n} \phi^i(u^i) = \infty$.
Then, an optimal team policy exists. 
\end{theorem}


\proof First we note that for a family of probability measures, if the marginals of these measures are tight, then this family of joint probability measures with these tight marginals is also tight (see e.g., \cite[Proof of Thm. 2.4]{yukselSICON2017}). In our theorem, the realized control action sets are not compact valued, but the control policy space considered is still weakly compact due to the properties of $\phi^i$: the marginals on $\mathbb{U}^i$ can be restricted to a weakly compact space.  

Observe that {\bf Step (ii)} of the proof of Theorem \ref{existenceControlTopology} satisfies the lower semi-continuity property using the fact that the cost is non-negative valued and through a truncation argument and letting the truncation variable go to infinity through a monotone convergence theorem (applied to the right hand side, and the truncated function replaced with the limit on the left hand side as a universal upper bound). That is, with $(\pi_N^k,\gamma^N_k) \to (\pi_N, \gamma^N)$, we have that for every $L \in \mathbb{R}_+$
\begin{eqnarray}
&& \liminf_{k \to \infty} J_N(\pi_N^k,\gamma^N_k) \nonumber \\
&& = \liminf_{k \to \infty} \int \pi_{N-1}^k(dx_{N-1}) Q(dy^N) \gamma^N_k(du^N|y^N) c(x_{N-1},y^N,u^N) \nonumber \\
&& \geq \liminf_{k \to \infty} \int \pi_{N-1}^k(dx_{N-1}) Q(dy^N) \gamma^N_k(du^N|y^N) \min(L,c(x_{N-1},y^N,u^N)) \nonumber \\
&& = \int \pi_{N-1}(dx_{N-1}) Q(dy^N) \gamma^N(du^N|y^N) \min(L,c(x_{N-1},y^N,u^N)),
\end{eqnarray}
where the last equality holds due to {\bf Step (ii)} of the proof of Theorem \ref{existenceControlTopology}. Since $L$ is arbitrary, 
\begin{eqnarray}
&& \liminf_{k \to \infty} J_N((\pi_N^k,\gamma^N_k)) \nonumber \\
&& \quad \quad \geq \int \pi_{N-1}(dx_{N-1}) Q(dy^N) \gamma^N(du^N|y^N) c(x_{N-1},y^N,u^N),
\end{eqnarray}
leading to the lower semi-continuity of $J_N(\pi_N, \gamma^N)$. 

Furthermore, steps similar to {\bf Step (iii)} of the proof of Theorem \ref{existenceControlTopology} can be adjusted to ensure that 
\[J_N(\pi_N):=\inf_{\gamma^N \in \Gamma^N} J_N(\pi_N, \gamma^N)\]
is lower semi-continuous in $\pi_{N}$. This follows from the fact that, for compact $\mathbb{U}$, the function
\[S(x) := \inf_{u \in \mathbb{U}} d(x,u),\]
is lower semi-continuous if $d(x,u)$ is. The argument follows from the following. Let $x_n \to x$ and $u^*_n$ be an optimal action at $x_n$, and $u^*$ be optimal at $x$. Then
\[\liminf_{n \to \infty} d(x_n,u^*_n) - d(x,u^*)\geq 0\]
for otherwise if
\[\liminf_{n \to \infty} d(x_n,u^*_n) - d(x,u^*) < 0\]
there exists a subsequence and $\delta > 0$ such that
\[\liminf_{n_k \to \infty} d(x_{n_k},u^*_{n_k}) - d(x,u^*) < -\delta \]
But since $d$ is lower semi-continuous, the above implies that for some further subsequence along which $u^*_{n'_k} \to v^*$ for some $v^*$
\[ d(\lim_{n'_k \to \infty} x_{n'_k}, u^*_{n'_k}) - d(x,u^*)  \leq  \liminf_{n_k \to \infty} d(x_{n'_k},u^*_{n'_k}) - d(x,u^*) < - \delta\]
But, $d(\lim_{n'_k \to \infty} x_{n'_k}, u^*_{n'_k}) = d(x,v^*)$ cannot be less than $d(x,u^*) - \delta$ since $u^*$ is optimal. Thus, at each step of the induction, the lower semi-continuity of the value function is presented. Since in the proof of Theorem \ref{existenceControlTopology}, $\pi_n \to \pi$ and $\gamma^i_n \to \gamma^i$ satisfy the identical conditions of lower semi-continuity and compactness, the result follows. 

Furthermore, an optimal selection exists for $\gamma^N$ through the measurable selection conditions \cite[Theorem 2]{himmelberg1976optimal}, \cite{Schal} \cite{kuratowski1965general}.

Since $J_N(\pi_N)$ is lower semi-continuous, {\bf Step (iv)} of the proof of Theorem \ref{existenceControlTopology} identically ensures that $J_{N-1}(\pi_{N-1},\gamma^{N-1})$ is also lower semi-continuous. 
The recursion then applies until $t=0$.  
\endproof


We next present a strategic measures based approach. In the following, we show that existence can be established even if the continuity is {\it only in the actions} (and not necessarily in the measurement variables). This applies both for static teams, or the static reduction of dynamics teams, with independent measurements. 

\begin{theorem}\label{existenceRelaxed3}
Consider a static or a dynamic team that admits a reduced form with independent measurements and $c_s$ (see (\ref{c_sDefn})) continuous only in ${\bf u}$ for every fixed $\omega_0, {\bf y}$. Suppose further that $\mathbb{U}^i$ is $\sigma-$compact (that is, $\mathbb{U}^i=\cup_n K_n$ for a countable collection of increasing compact sets $K_n$) and without any loss, the control laws can be restricted to those with $E[\phi^i(u^i)]\leq M$ for some lower semi-continuous $\phi^i: \mathbb{U}^i \to \mathbb{R}_+$ which satisfies $\lim_{n \to \infty} \inf_{u^i \notin K_n} \phi^i(u^i) = \infty$.
Then, an optimal team policy exists. 
\end{theorem}
\proof Consider the product measure: 
\[P(d\omega_0) \prod_{k=1}^N(Q^k\gamma^k)(dy^k,du^k).\]
Suppose that every action $\hat{\gamma}^k$, that is, $(Q^k\gamma^k)_m(dy^k,du^k)$, converges to $(Q^k\gamma^k)(dy^k,du^k)$ weakly. By \cite[p. 57]{Par67}, the joint product measure $P(d\omega_0) \prod_{k=1}^N(Q^k\gamma^k)_m(dy^k,du^k)$
will converge weakly to $P(d\omega_0) \prod_{k=1}^N(Q^k\gamma^k)(dy^k,du^k)$; see also \cite[Section 5]{serfozo1982convergence}.
Recall the $w$-$s$ topology introduced by Sch\"al \cite{schal1975dynamic}: The $w$-$s$ topology on the set of probability measures ${\cal P}(\mathbb{X} \times \mathbb{U})$ is the coarsest topology under which $\int f(x,u) \nu(dx,du): {\cal P}(\mathbb{X} \times \mathbb{U}) \to \mathbb{R}$ is continuous for every measurable and bounded $f$ which is continuous in $u$ for every $x$ (but unlike weak topology, $f$ does not need to be continuous in $x$). Since the marginals on $\prod_k \mathbb{Y}^k$ is fixed, \cite[Theorem 3.10]{schal1975dynamic} (or  \cite[Theorem 2.5]{balder2001}) establishes that the set of all probability measures with a fixed marginal on $\prod_k \mathbb{Y}^k$ is relatively compact under the $w$-$s$ topology. This in turn ensures that the function (by a truncation and then a limiting argument as done earlier)
\[\int P(d\omega_0) \prod_{k=1}^N(Q^k\gamma^k)(dy^k,du^k) c(\omega_0,{\bf u}),\]
is lower semi-continuous under the $w$-$s$ topology. Since the set of admissible strategic measures is sequentially compact under the $w$-$s$ topology,  existence of an optimal team policy follows. The proof for dynamic case follows analogously.
\endproof

\begin{remark}\label{detRemark}
Building on \cite[Theorems 2.3 and 2.5]{YukselSaldiSICON17} and \cite[p. 1691]{gupta2014existence} (due to Blackwell's irrelevant information theorem \cite{Blackwell2,Blackwell3}, \cite[p. 457]{YukselBasarBook}), an optimal policy, when exists, can be assumed to be {\it deterministic}.
\end{remark}

\subsection{Comparison with the previously reported existence results and a refinement}

Existence of optimal policies for static and a class of sequential dynamic teams have been studied recently in \cite{gupta2014existence,YukselSaldiSICON17}. More specific setups have been studied in \cite{WuVer11}, \cite{wit68}, \cite{YukselOptimizationofChannels} and \cite{YukselBasarBook}. Existence of optimal team policies has been established in \cite{charalambous2017centralizedI} for a class of continuous-time decentralized stochastic control problems. For a class of teams which are convex, one can reduce the search space to a smaller parametric class of policies, such as linear policies for quasi-classical linear quadratic Gaussian problems \cite{rad62,KraMar82,HoChu}.

\begin{theorem}\label{existenceT}\cite{YukselSaldiSICON17}
(i) Consider a static or dynamic team. Let the loss function $c$ be lower semi-continuous in $(\omega_0, {\bf u})$ and $L_R(\mu)$ be a compact subset under weak topology. Then, there exists an optimal team policy. This policy is deterministic and hence induces a strategic measure in $L_A(\mu)$.  \\
(ii) Consider a static team or the static reduction of a dynamic team with $c$ denoting the loss function. Let $c$ be lower semi-continuous in $\omega_0, {\bf u}$ and $L_C(\mu)$ be a compact subset under weak topology. Then, there exists an optimal team policy. This policy is deterministic and hence induces a strategic measure in $L_A(\mu)$.  \\
\end{theorem}

However, we recall that unless certain conditions are imposed, the conditional independence property is not preserved under weak or setwise convergence and thus $L_R(\mu)$ is in general not compact; see \cite[Theorem 2.7]{YukselSaldiSICON17}.
 
We refer the reader to \cite{barbie2014topology}, and the references therein, for further related results on such intricacies on conditional independence properties. A sufficient condition for compactness of $L_R$ under the weak convergence topology was reported in \cite{gupta2014existence}; we re-state this result in a brief and different form below for reader's convenience:
\begin{theorem}\label{SuffCon1} \cite{gupta2014existence}
Consider a static team where the action sets $\mathbb{U}^i, i \in {\cal N}$ are compact. Furthermore, if the measurements satisfy
\[P(d{\bf y}|\omega_0) = \prod_{i=1}^n Q^i(dy^i|\omega_0),\]
where $Q^i(dy^i|\omega_0) = \eta^i(y^i,\omega_0) \nu^i(dy^i)$ for some positive measure $\nu^i$ and continuous $\eta^i$ so that for every $\epsilon >0$, there exists $\delta > 0$ so that for $\rho_i(a,b)<\delta$ (where $\rho_i$ is a metric on $\mathbb{Y}^i$)
\[|\eta^i(b,\omega_0) - \eta^i(a,\omega_0) | \leq \epsilon h^i(a,\omega_0),\]
with $\sup_{\omega_0} \int h^i(a,\omega_0) \nu^i(dy^i) < \infty$, then the set $L_R(\mu)$
 is weakly compact and if $c(\omega_0,{\bf u})$ is lower semi-continuous, there exists an optimal team policy (which is deterministic and hence in $L_A(\mu)$).
  \end{theorem}

The results in \cite{gupta2014existence} also apply to static reductions for sequential dynamic teams, and a class of teams with unbounded cost functions and non-compact action spaces that however satisfies some moment-type cost functions leading to a tightness condition on the set of strategic measures leading to a finite cost. In particular, the existence result applies to the celebrated counterexample of Witsenhausen \cite{wit68}. 

Theorems \ref{existenceControlTopology} and \ref{existenceRelaxed2} provide weaker conditions when compared with Theorem \ref{SuffCon1}.

For a class of sequential teams with perfect recall, we established the existence of optimal team policies through Theorem \ref{StrategicCharacterization} in \cite{YukselSaldiSICON17}. Note that the cost function is given by $c(\omega_0, {\bf u})$, where $\omega_0$ is an exogenous random variable. 
\begin{theorem}\label{SuffCon2}\cite[Theorem 2.9]{YukselSaldiSICON17}
Consider a sequential team with a classical information structure with the further property that $\sigma(\omega_0) \subset \sigma(y^1)$ (under every policy, so that $y^1$ contains $\omega_0$). Suppose further that $\prod_k (\mathbb{Y}^k \times \mathbb{U}^k)$ is compact. If $c$ is lower semi-continuous and each of the kernels $p_n$ (defined in (\ref{kernelDefn})) is weakly continuous so that
\begin{eqnarray}\label{weakConKer}
\int f(y^n) p_n(dy^n | \omega_0,y^{1},\ldots,y^{n-1},u^{1},\cdots,u^{n-1})
\end{eqnarray} 
is continuous in $\omega_0,u^{1},\cdots,u^{n-1}$ for every continuous and bounded $f$, there exists an optimal team policy which is deterministic.
\end{theorem}

A further existence result along similar lines, for a class of static teams, is presented next. This result refines \cite[Theorem 3.5]{YukselSaldiSICON17}, by relaxing the weak continuity condition there on the kernel (i.e., the condition that $\int f(y^n) P(dy^n | y^{n-1})$ is continuous in $y^{n-1}$ for every continuous and bounded $f$).

\begin{theorem}\label{SuffCon2'''}
Consider a static team with a classical information structure (that is, with an expanding information structure so that $\sigma(y^n) \subset \sigma(y^{n+1}), n \geq 1$). Suppose further that $\prod_k (\mathbb{Y}^k \times \mathbb{U}^k)$ is compact. If 
\[\tilde{c}(y^1,\cdots,y^N,u^1,\cdots,u^N):=E[c(\omega_0,{\bf u}) | {\bf y}, {\bf u}]\]
is jointly lower semi-continuous in ${\bf u}$ for every ${\bf y}$, there exists an optimal team policy which is deterministic.
\end{theorem}
\proof Different from Theorem \ref{SuffCon2}, we eliminate the use of $\omega_0$, and study the properties of the set of strategic measures. Different from \cite[Theorem 3.5]{YukselSaldiSICON17}, we relax weak continuity. Once again, instead of the weak topology, we will use the $w$-$s$ topology \cite{Schal}.

As in the proof of Theorem \ref{SuffCon2}, when $\prod_k \mathbb{Y}^k \times \mathbb{U}^k$ is compact, the set of all probability measures on $\prod_k \mathbb{Y}^k \times \mathbb{U}^k$ forms a weakly compact set. Since the marginals on $\prod_k \mathbb{Y}^k$ is fixed, \cite[Theorem 3.10]{Schal} (or  \cite[Theorem 2.5]{balder2001}) establishes that the set of all probability measures with a fixed marginal on $\prod_k \mathbb{Y}^k$ is relatively compact under the $w$-$s$ topology. Therefore, it suffices to ensure the closedness of the set of strategic measures, which leads to the sequential compactness of the set under this topology. To facilitate such a compactness condition, as earlier we first {\it expand} the information structure so that DM $k$ has access to all the previous actions $u^{1},\cdots,u^{k-1}$ as well. Later on, we will see that this expansion is redundant. With this expansion, any $w$-$s$ converging sequence of strategic measures will continue satisfying (\ref{convD2}) in the limit due to the fact that there is no conditional independence property in the sequence since all the information is available at DM $k$. That is, $P_n(du^n | y^n,y_{[0,n-1]},u_{[0,n-1]})$ satisfies the conditional independence property trivially as all the information is available. On the other hand, for each element in the sequence of conditional probability measures, the conditional probability for the measurements writes as
$P(dy^n | y_{[0,n-1]},u_{[0,n-1]}) = P(dy^n | y^{n-1})$.  We wish to show that this also holds for the $w$-$s$ limit measure. Now, we have that for every $n$, $y^n \leftrightarrow y^{n-1} \leftrightarrow h_{n-1}$ forms a Markov chain. By considering the convergence properties only on continuous functions and bounded $f$, as in (\ref{convD1}), with $P_m \to P$ weakly, we have that
\begin{eqnarray}
&&\int P(dy^n | y^{n-1}) P_m(dh_{n-1}) f(y^n,h_{n-1}) = \int \bigg( P(dy^n | y^{n-1})  f(y^n,h_{n-1}) \bigg) P_m(dh_{n-1}) \nonumber \\
&& \qquad \to  \int \bigg( P(dy^n | y^{n-1})  f(y^n,h_{n-1}) \bigg) P(dh_{n-1}) =  \int P(dy^n | y^{n-1}) P(dh_{n-1}) f(y^n,h_{n-1}) \nonumber
\end{eqnarray}
Here, $\bigg( P(dy^n | y^{n-1})   f(y^n,h_{n-1}) \bigg)$ is not continuous in $y^{n-1}$, but it is in $h_{n-2}$ and $u^{n-1}$ by an application of the dominated convergence theorem, and since \[P_m(dy^1,\ldots,dy^{n-1},du^{1},\ldots,du^{n-1}) \to P(dy^1,\ldots,dy^{n-1},du^{1},\ldots,du^{n-1}),\]
 in the $w$-$s$ sense (setwise in the measurement variable coordinates), convergence holds. 
Thus, (\ref{convD1}) is also preserved.
 Hence, for any $w$-$s$ converging sequence of strategic measures satisfying (\ref{convD1})-(\ref{convD2}) so does the limit since the team is static and with perfect-recall. By \cite[Theorem 3.7]{Schal}, and the generalization of Portmanteau theorem for the $w$-$s$ topology, the lower semi-continuity of $\int \mu(d{\bf y}, d{\bf u}) \tilde{c}({\bf y},{\bf u})$ over the set of strategic measures leads to the existence of an optimal strategic measure. As a result, the existence follows from similar steps to that of Theorem \ref{existenceT}. Now, we know that an optimal policy will be deterministic (see Remark \ref{detRemark}). Thus, a deterministic policy may not make use of randomization, therefore DM $k$ having access to $\{y^{k},y^{k-1},y^{k-2}, \cdots\}$ is informationally equivalent to him having access to $\{y^{k},(y^{k-1},u^{k-1}),(y^{k-2},u^{k-2})\}$ for an optimal policy. Thus, an optimal team policy exists.
\endproof


\subsection{Discussion}
We thus report that Theorem \ref{existenceRelaxed3} (for static teams or dynamic teams with an independent-measurements reduction) and Theorems \ref{SuffCon2} and \ref{SuffCon2'''} (for sequential teams that do not allow an independent-measurements reduction) are the most general existence results in the literature, to our knowledge, for sequential team problems. These results complement each other and cover a large class of decentralized stochastic control problems. 

In view of completeness, we also note the following: Recently, \cite{SaldiArXiv2017} presented an alternative proof approach for the existence problem and arrived at existence results through constructing a new topology on the space of team policies; the existence result presented in \cite{SaldiArXiv2017} is implied by Theorem \ref{existenceRelaxed3}. Another recent work \cite{gupta2020existence} has established existence results for a setup where either the measurements are countable or there is a common information among decision makers which is countable space-valued with the private information satisfying an absolute continuity condition; as noted earlier in the paper, static reduction applies in both such setups and the results presented in this paper (notably Theorem \ref{existenceRelaxed3}) generalize those reported in \cite{gupta2020existence}. We note also that the use of $w$-$s$ topology used in Theorem \ref{existenceRelaxed3} significantly relaxes the requirements of continuity.


\subsection{Some applications and revisiting classical existence results for single-DM stochastic control}  
\subsubsection{Witsenhausen's counterexample with Gaussian variables}\label{WitsenSection}
Consider the celebrated Witsenhausen's counterexample \cite{wit68} as depicted in Figures \ref{chap3figWitsenOrig} and \ref{chap3figWitsen}: This is a dynamic non-classical team problem with $y^1$ and $w^1$ zero-mean independent Gaussian random variables with unit variance and $u^1=\gamma^1(y^1)$, $u^2 = \gamma^2(u^1+w^1)$ and the cost function $c(\omega,u^1,u^2)= k(y^1-u^1)^2 + (u^1-u^2)^2$ for some $k > 0$. Witsenhausen's counterexample can be expressed, through a change of measure argument (also due to Witsenhausen) as follows. The static reduction proceeds as follows:
\begin{eqnarray}\label{stW}
&&\int (k(u^1-y^1)^2 + (u^1-u^2)^2) Q(dy^1)\gamma^1(du^1|y^1)\gamma^2(du^2|y^2)P(dy^2|u^1) \nonumber \\
&& = \int (k(u^1-y^1)^2 + (u^1-u^2)^2) Q(dy^1) \gamma^1(du^1|y^1)\gamma^2(du^2|y^2)\eta(y^2-u^1) dy^2 \nonumber \\
&& = \int \bigg( (k(u^1-y^1)^2 + (u^1-u^2)^2) \gamma^1(du^1|y^1)\gamma^2(du^2|y^2) {\eta(y^2-u^1)  \over \eta(y^2)} \bigg) Q(dy^1) \eta(y^2)dy^2 \nonumber \\
&& = \int \bigg( (k(u^1-y^1)^2 + (u^1-u^2)^2) \gamma^1(du^1|y^1)\gamma^2(du^2|y^2) {\eta(y^2-u^1)  \over \eta(y^2)} \bigg) Q(dy^1) Q(dy^2) \nonumber
\end{eqnarray}
where $Q$ denotes a Gaussian measure with zero mean and unit variance and $\eta$ its density. Since the optimal policy for $\gamma^2(y^2) = E[u^1|y^1]$ and $E[(E[u^1|y^1])^2] \leq E[(u^1)^2]$, it is evident with a two-stage analysis (see \cite[p. 1701]{gupta2014existence}) that without any loss we can restrict the policies to be so that $E[(u^i)^2] \leq M$ for some finite $M$, for $i=1,2$; this ensures a weak compactness condition on both $\hat{\gamma}^1$ and $\hat{\gamma}^2$. Since the reduced cost $\bigg( (k(u^1-y^1)^2 + (u^1-u^2)^2) {\eta(y^2-u^1)  \over \eta(y^2)} \bigg)$ is continuous in the actions, Theorem \ref{existenceRelaxed2} applies.

%

\subsubsection{Existence for partially observable Markov Decision Processes (POMDPs)}

Consider a partially observable stochastic control problem (POMDP) with the following dynamics.
$$x_{t+1}=f(x_t,u_t,w_t), \quad \quad y_{t}=g(x_t,v_t).$$
Here, $x_t$ is the $\mathbb{X}$-valued state, $u_t$ is the $\mathbb{U}$-valued the control, $y_t$ is the $\mathbb{Y}$-valued measurement process. In this section, we will assume that these spaces are finite dimensional real vector spaces. Furthermore, $(w_t,v_t)$ are i.i.d noise processes and $\{w_t\}$ is independent of $\{v_t\}$. The controller only has causal access to $\{y_t\}$: A deterministic admissible control policy $\Pi$ is a sequence of functions $\{\gamma_t\}$ so that $u_t = \gamma(y_{[0,t]};u_{[0,t-1]})$. 
The goal is to minimize
\[E^{\Pi}_{x_0} [\sum_{t=0}^{T-1}  c(x_t,u_t)],\]
for some continuous and bounded $c: \mathbb{X} \times \mathbb{U} \to \mathbb{R}_+$. 

Such a problem can be viewed as a decentralized stochastic control problem with increasing information, that is, one with a {\it classical} information structure.

Any POMDP can be reduced to a (completely observable) MDP \cite{Yus76}, \cite{Rhe74}, whose states are the posterior state distributions or {\it beliefs} of the observer. A standard approach for solving such problems then is to reduce the partially observable model to a fully observable model (also called the belief-MDP) by defining
\[\pi_t(A):=P(x_t \in A | y_{[0,t]},u_{[0,t-1]}), A \in {\cal B}(\mathbb{X})\]
and observing that $(\pi_t, u_t)$ is a controlled Markov chain where $\pi_t$ is ${\cal P}(\mathbb{X})$-valued with ${\cal P}(\mathbb{X})$ being the space of probability measures on $\mathbb{X}$ under the weak convergence topology. Through such a reduction, existence results can be established by obtaining conditions which would ensure that the controlled Markovian kernel for the belief-MDP is weakly continuous, that is if $\int F(\pi_{t+1}) P(d\pi_{t+1} | \pi_t = \pi, u_t= u)$ is jointly continuous (weakly) in $\pi$ and $u$ for every continuous and bounded function $F$ on ${\cal P}(\mathbb{X})$. 

This was studied recently in \cite[Theorem 3.7, Example 4.1]{FeKaZg14} and \cite{KSYWeakFellerSysCont} (see also \cite{budhiraja2002invariant} in a control-free context). In the context of the example presented, if $f(\cdot,\cdot,w)$ is continuous and $g$ has the form: $y_t = g(x_t) + v_t$, with $g$ continuous and $w_t$ admitting a continuous density function $\eta$, an existence result can be established building on the measurable selection criteria under weak continuity \cite[Theorem 3.3.5, Proposition D.5]{HernandezLermaMCP}, provided that $\mathbb{U}$ is compact. 

On the other hand, through Theorem \ref{SuffCon2'''}, such an existence result can also be established by obtaining a static reduction under the aforementioned conditions. Indeed, through (\ref{c_sDefn}), with $\eta$ denoting the density of $v_n$, we have $P(y_n \in B | x_n) = \int_B \eta(y - g(x_n)) dy$. With $\eta$ and $g$ continuous and bounded, taking $y^n := y_n$, by writing $x_{n+1} =f(x_{n},u_n,w_n) = f(f(x_{n-1},u_{n-1},w_{n-1}),u_n,w_n)$, and iterating inductively to obtain \[x_{n+1} = h_n(x_0,{\bf u}_{[0,n-1]}, {\bf w}_{[0,n-1]}),\] for some $h_n$ which is continuous in ${\bf u}_{[0,n-1]}$ for every fixed $x_0, {\bf w}_{[0,n-1]}$, one obtains a reduced cost (\ref{c_sDefn}) that is a continuous function in the control actions. Theorem \ref{SuffCon2'''} then implies the existence of an optimal control policy. \sy{This reasoning is also applicable when the measurements are not additive in the noise but with $P(y_n \in B | x_n=x) = \int_B f(y, x) \eta(dy)$ for some $f$ continuous in $x$ and $\eta$ a reference measure.}

\subsubsection{Revisiting fully observable Markov Decision Processes with the construction presented in the paper}

Consider a fully observed Markov decision process where the goal is to minimize
\[E^{\Pi}_{x_0} [\sum_{t=0}^{T-1}  c(x_t,u_t)],\]
for some continuous and bounded $c: \mathbb{X} \times \mathbb{U} \to \mathbb{R}_+$. Suppose that the controller has access to $x_{[0,t]}, u_{[0,t-1]}$ at time $t$. This system can always be viewed as a sequential team problem with a classical information structure. Under the assumption that the transition kernel according to the usual formulation, that is $P(dx_1 | x_0=x,u_0=u)$ is weakly continuous (in the sense discussed in the previous application above), it follows that the transition kernel according to the formulation introduced in Section \ref{newDPF} is also weakly continuous by an application \cite[Theorem 3.5]{serfozo1982convergence}. It follows that the condition in Lemma \ref{suffCondUpperSemiCont} holds when $\mathbb{U}$ is compact, and hence the existence of an optimal policy follows. A similar analysis is applicable when one considers the case where $P(dx_1 | x_0=x,u_0=u)$ is strongly continuous in $u$ for every fixed state $x$ and the bounded cost function is continuous only in $u$ (this is another typical setup where measurable selection conditions hold \cite[Chapter 3]{HernandezLermaMCP}).

\section{Conclusion}
For sequential dynamic teams with standard Borel measurement and control action spaces, we introduced a general controlled Markov and dynamic programming formulation, established its well-posedness and provided new existence results for optimal policies. Our dynamic program builds on, but significantly modifies, Witsenhausen's standard form, and the formulation here through our recent work \cite{YukselSaldiSICON17} allows for a standard Borel controlled state dynamics. This leads to a well-defined dynamic programming recursion for a very general class of sequential decentralized stochastic control problems. In addition, refined existence results have been obtained for optimal policies; these generalize prior results in the literature.

\section{Acknowledgements}
The author is grateful to the reviewers for their comments which have helped improve the paper. The author is also grateful to Prof. Naci Saldi and Sina Sanjari for the detailed suggestions and corrections they have noted. The author is also grateful to Prof. Tamer Ba\c{s}ar for many discussions, including those during the earlier stages of \cite{YukselBasarBook}; the counterexample presented in \cite{YukselBasarBook} used in Section \ref{counterExampleSection} is due to him.  

\bibliographystyle{amsplain}

\end{document}